\documentclass[11pt,a4paper]{article}

\usepackage{amsfonts,amsmath,layout}
\usepackage{mathrsfs}  
\usepackage{amssymb,mathabx,amsfonts,amsthm,amscd,stmaryrd,dsfont,esint,upgreek,constants,todonotes}

\usepackage[utf8]{inputenc}
\usepackage{makeidx}
\usepackage[T1]{fontenc}
\usepackage{amsmath}
\usepackage{amsthm}
\usepackage{amsfonts}
\usepackage{amssymb}
\usepackage{alltt}
\usepackage{color}
\usepackage{graphicx}

\newcommand{\Z}{\mathbb Z}

\newcommand{\R}{\mathbb R}

\def\Pk{{\mathcal P}_1}

\def\dk{{\bf d}_1}

\newcommand{\be}{\begin{equation}}
\newcommand{\ee}{\end{equation}}
\def\1{{\bf 1}}

\def\ep{\epsilon}

\def\Dt0{{\bf D}(t_0)}

\def\E{{\bf E}}

\def\to{\rightarrow}

\def\ds{\displaystyle}

\def\ds{\displaystyle}

\def \to{\rightarrow}
\def \states{\mathbb{T}^d}
\def \T{\mathbb{T}}
\def \R {\mathbb{R}}

\def \Z {\mathbb{Z}}
\def \mes{\mathcal{P}}
\def\Pk{\mes(\states)}

\def\dk{{\bf d}_1}
\def\dive{{\rm div}}

\def \ep{\varepsilon}

\def\E{\mathbb E}

\def\inte{\int_{\T^d}}
\definecolor{ProcessBlue}{cmyk}{1,0,0,0.40}

\newtheorem{Theorem}{Theorem}[section]

\newtheorem{Proposition}[Theorem]{Proposition}
\newtheorem{Lemma}[Theorem]{Lemma}

\newtheorem{Remark}{Remark}[section]

\newtheorem{Example}{Example}[section]

\title{On the (in)efficiency of MFG equilibria}
\author{Pierre Cardaliaguet\thanks{Universit\'{e} Paris-Dauphine, PSL Research University, Ceremade. cardaliaguet@ceremade.dauphine.fr} \and Catherine Rainer\thanks{Universit\'{e} de Bretagne Occidentale, LMBA. Catherine.Rainer@univ-brest.fr
}}

\begin{document}

\maketitle




\begin{abstract} Mean field games (MFG) are dynamic games with infinitely many infinitesimal agents. In this context, we study the efficiency of Nash MFG equilibria: Namely, we compare the social cost of a MFG equilibrium with the minimal cost a global planner can achieve. We find a structure condition on the game under which there exists efficient MFG equilibria and, in case this condition is not fulfilled, quantify how inefficient  MFG equilibria are. 
\end{abstract}

\section{Introduction}

Mean field games (MFG) study Nash equilibria in differential games with infinitely many  indistinguishable agents. In this note, we investigate the classical question of the efficiency for these Nash equilibria:  we compare the social cost corresponding to a MFG equilibrium with the optimal social cost obtained by a global planner.  

To fix the ideas and describe the model we have in mind, we start with a  finite horizon differential game played by a large number of agents (say $N$). Agent $i\in \{1, \dots, N\}$ controls her dynamics: 
$$
\left\{\begin{array}{l}
dX^i_t=\alpha^i_t dt + \sqrt{2}dB^i_t , \qquad t\in [0,T],\\
X^i_0 = x^i_0
\end{array}\right.
$$
where $T$ is the horizon, the $(B^i)$ are $N$ independent Brownian motions and  the $(x^i_0)$ are  $N$ i.i.d. random variables on $\R^d$,  independent of the $(B^i)$,  with law $m_0$. 
The control $(\alpha^i)$ is chosen by agent $i$ in order to minimize a  cost of the form
$$
J^i(\alpha^1, \dots, \alpha^N) = \E\left[\int_0^T L(X^i_t, \alpha^i_t,m^N_{{\bf X}_t}) \ dt+G(X^i_T,m^N_{{\bf X}_T})\right],
$$
where 
$m^N_{{\bf X}_t}=\frac{1}{N}\sum_{j=1}^N\delta_{X^j_t}$ is the empirical measure of the players. We assume that dynamics and costs have a special structure: agent $i$ controls directly her own drift and her running cost at time $t$ depends on her position $X^i_t$, on her control $\alpha^i_t$ and on the empirical measure of all players $m^N_{{\bf X}_t}$; her terminal cost depends on her position $X^i_T$ at the terminal time $T$ and on the empirical measure $m^N_{{\bf X}_T}$ at that time. Note that, under our assumptions, the agents have symmetric dynamics and costs functions. 

The social cost associated with the $N$ agents is the average of the $J^i$: 
$$
J(\alpha^1, \dots, \alpha^N):= \frac{1}{N} \sum_{i=1}^N J^i(\alpha^1, \dots, \alpha^N). 
$$

{\it Global planning.} If there is a global planner, the cost to minimize is $J$, over the adapted controls $(\alpha^1, \dots, \alpha^N)$. This minimum could be computed by standard tools (for instance, a Hamilton-Jacobi equation in $\R^{Nd}$). However, we are  interested in the value of the minimum when the number of agents is large. It is proved in \cite{lacker2017limit}, in a much more general context, that 
$$
\lim_{N\to+\infty} \inf_{\alpha^1, \dots, \alpha^N} J(\alpha^1, \dots, \alpha^N) = {\mathcal C}^*,
$$
where ${\mathcal C}^*$ is the cost associated with the McKean-Vlasov control problem: 
$$
{\mathcal C}^*:= \inf_{(m,\alpha)} \int_0^T\int_{\R^d}  L(x,\alpha(t,x),m(t)) m(t,x)\  dxdt+ \int_{\R^d} G(x,m(T))m(T, x)dx
$$
and where the infimum is taken over the pairs $(m,\alpha)$ such that 
$$
\partial_t m -\Delta m+\dive(m\alpha)=0,\qquad 
m(0,x)=m_0(x).
$$
The quantity ${\mathcal C}^*$ is  our first main object of investigation. We interpret it as the {\it social cost associated with a global planner}. 

Following \cite{lasry2007mean} (see also \cite{bensoussan2013mean, briani2018stable} and Lemma \ref{lem.OC} below),  and under suitable conditions stated below, the above problem for ${\mathcal C}^*$ has a minimum $(\hat m,\hat \alpha)$ and there exists a map $\hat u$ such that $(\hat u,\hat m)$ solves the forward-backward system 
\be\label{eq:MFGcnintro}
\left\{ \begin{array}{l}
\ds -\partial_t \hat u -\Delta \hat u+H(x,D\hat u,\hat m(t))= \int_{\R^d} \frac{\delta L}{\delta m}(y,\hat \alpha(t,y), x, \hat m(t))\hat m(t,y)dy\; {\rm in}\; (0,T)\times \R^d\\
\ds \partial_t \hat m -\Delta \hat m- \dive(\hat mD_pH(x,D\hat u(t,x),\hat m(t)))=0 \qquad {\rm in}\; (0,T)\times \R^d\\
\ds \hat m(0,x)=m_0(x), \; \hat u(T,x)= \frac{\delta \widehat{{\mathcal G}}}{\delta m}(\hat m(T),x)\qquad {\rm in}\; \R^d
\end{array}\right.
\ee
where  $H(x,p,m)=\sup_{\alpha\in \R^d} -\alpha\cdot p -L(x,\alpha,m)$ is the convex conjugate of $L$ and where we have denoted by  
$$
\hat \alpha(t,x)=- D_pH(x,D\hat u(t,x),\hat m(t))
$$
the optimal feedback control of the global planner. The map $\widehat{{\mathcal G}}$ is  defined  by 
\be\label{defMathcalFGIntro}
\widehat{{\mathcal G}}(m):= \int_{\R^d} G(x,m)m(dx)
\ee
while $\delta L/\delta m$ and $\delta \hat {\mathcal G}/\delta m$ are the derivatives of the maps $m\to L(x,\alpha,m)$ and  $m\to \hat {\mathcal G}(m)$, respectively, with respect to the measure variable $m$  (see Section \ref{sec:2}). 
\bigskip

{\it Decentralized setting.} When there is no cooperation between the agents, one expects them to play a Nash equilibrium. The characterization of Nash equilibria (in memory strategy) is known in this setting \cite{buckdahn2004nash, kononenko1976equilibrium} and related to the Folk's Theorem (any feasible and individually rational payoff can be achieved as a Nash equilibrium). However, when the number $N$ of agents is large and the agents are indistinguishable, it is not reasonable to ask all the agents to observe each other: the notion of memory strategy (or even of global feedback strategy) does not seem to make much sense. One would expect the agent to act instead by taking into account their own position and the distribution of the position of other agents: this is precisely what  mean field games formalize. 

\bigskip

{\it Mean field games.} Mean field games (MFG) model differential games with infinitely many indistinguishable players. They were introduced  by Lasry and Lions \cite{lasry2006jeuxi, lasry2006jeuxii, lasry2007mean}. At the same period, Huang, Caines and Malham\'e discussed the same concept under the terminology of ``Nash certainty equivalence principle" \cite{huang2003individual, huang2006large}. The MFG system associated with the above control  problem reads, in terms of PDEs, 
\be\label{MFGintro}
\left\{ \begin{array}{l}
-\partial_t u -\Delta u+H(x,Du, m(t))=0\qquad {\rm in}\; (0,T)\times \R^d\\
\partial_t m -\Delta m-\dive(mD_pH(x,Du,m(t)))=0\qquad {\rm in}\; (0,T)\times \R^d\\
m(0,x)=m_0(x), \; u(T,x)= G(x,m(T))\qquad {\rm in}\;  \R^d.
\end{array}\right.
\ee
 In the above system,  $u=u(t,x)$ is the value function of a typical player while $m=m(t,x)$ describes the evolving probability density of all agents. Note that the drift $-D_pH(x, Du(t,x))$ in the equation for $m$ corresponds to the optimal feedback of the agent. Heuristically, the pair $(u,m)$ describes a Nash equilibrium in the infinite population problem. 
\bigskip

The {\it social cost associated with a MFG equilibrium} $(u,m)$, which is the averaged cost of each player, can be defined  as: 
$$
{\mathcal C}(u,m):= \int_0^T\int_{\R^d}  L(x,\alpha^*(t,x),m(t)) m(t,x)\ dxdt + \int_{\R^d} G(x,m(T))m(T,x)dx,
$$
where $\alpha^*(t,x)= -D_pH(x,u(t,x))$ is the optimal feedback in the MFG. 
The quantity ${\mathcal C}(u,m)$ is the second main object of investigation of this paper. 
\bigskip 

{\it Comparison between the two problems.} The difference between the two problems---the centralized optimal control of McKean-Vlasov dynamics and the  MFG equililbria---has been often discussed in the literature: see, for instance, \cite{bensoussan2013mean, carmona2017probabilistic, carmona2015forward, carmona2013control,  huang2012social}.
So far the attention has focussed on the difference in structure between the two systems of equations (namely, for our problem, \eqref{eq:MFGcnintro} and \eqref{MFGintro}). Note that, in our specific setting, there is no real difference between \eqref{eq:MFGcnintro} and \eqref{MFGintro}: so one could expect that the two problems are very close in terms of social cost.

\bigskip 

{\it Comparison between ${\mathcal C}^*$ and ${\mathcal C}(u,m)$.} In this paper, we plan to compare the social costs ${\mathcal C}^*$ and ${\mathcal C}(u,m)$. Obviously one has ${\mathcal C}^*\leq {\mathcal C}(u,m)$. We want to understand a little better  the case of equality and the size of the difference ${\mathcal C}(u,m)-{\mathcal C}^*$. 

This question has been often addressed in the classical game theory: a characterization of efficiency can be found for instance in  \cite{dubey1986inefficiency}, which also proved that, generically, the Nash equilibria are not efficient. The problem became very popular  under the name of ``price of anarchy", introduced in \cite{koutsoupias1999worst} for noncooperative games in which agents share a common resource. We also refer for instance to  \cite{johari2005efficiency, johari2003network, roughgarden2002bad, roughgarden2004bounding} and the references therein, in the framework of  selfish routing games and congestion games. Related to our setting with infinitely many players, the recent paper \cite{lacker2017rare} discusses the price of anarchy for static games with a large number of players. 

This large literature is in sharp contrast with the literature on differential games, where efficiency has  seldom been investigated, and only recently:   \cite{bacsar2011prices}  estimates the price of anarchy in some scalar linear-quadratic (LQ) differential games; Directly related to our work, \cite{balandat2013efficiency} addresses the question of the inefficiency of MFG Nash equilibria by  numerical simulations. This question is also  discussed in \cite{CaGrTa2017}, in the settings of LQ MFG  and of MFG on finite Markov chains. 
    
\bigskip

{\it Main results.}  The main topic of our paper is the estimate of the difference between  ${\mathcal C}^*$ and ${\mathcal C}(u,m)$---in our  set-up, the ratio ${\mathcal C}^*/{\mathcal C}(u,m)$, generally used for the price of anarchy, does not seem to make much sense. To simplify a little the estimates, we work in the periodic setting (and therefore in the torus $\T^d=\R^d/\T^d$) instead of $\R^d$: We expect the similar results to hold for other boundary conditions, but the proof should be more technical. 

Our starting point is the obvious remark that the MFG system \eqref{MFGintro} describing $(u,m)$ and the system of necessary conditions \eqref{eq:MFGcnintro} characterizing the minimum for ${\mathcal C}^*$ are very close, and, in fact, are (almost) identical if 
$$
 \int_{\T^d} \frac{\delta L}{\delta m}(y,\hat \alpha(t,y), x, \hat m(t))\hat m(t,y)dy =0 \qquad {\rm and}\qquad
 \int_{\T^d} \frac{\delta G}{\delta m}(y,m(T),x)m(T,dy) =0.  
$$
It turns out that, if  a MFG equilibrium $(u,m)$ is efficient, i.e., if ${\mathcal C}(u,m)= {\mathcal C}^*$, then the  above equalities {\it must} hold (Proposition \ref{prop.eff}). 

When the above equalities do not hold, one may wonder how far MFG equilibria are from efficiency. This is precisely the aim of our main results (Theorem  \ref{thm:LB} and Theorem  \ref{thm:UB}) which give lower and upper bounds for the difference between ${\mathcal C}^*$ and ${\mathcal C}(u,m)$. The lower bound, stated in Theorem  \ref{thm:LB}, reads, for any $\ep>0$:  
 \begin{align*}
{\mathcal C}( u, m)-{\mathcal C}^*\geq & \; \;
 C_\ep^{-1}\Bigl(\int_\ep^{T-\ep} \inte \left[ \inte  \frac{\delta L}{\delta m}(x, \alpha^*(t,x),y,m(t)) m(t,x)dx\right]^2  dydt \Bigr)^2 \\
&\qquad \qquad\qquad\qquad  +   C^{-1}\Bigl(\inte \left[ \inte  \frac{\delta G}{\delta m}(x, m(T),y) m(T,x)dx\right]^2  dy \Bigr)^4,
\end{align*}
where $\alpha^*(t,x)=-D_pH(x,Du(t,x),m(t))$. The constants $C\geq 1$  depends on the regularity of the data and   $C_\ep\geq 1$ depends also on $\ep$. The presence of $\ep$ is technical and is related with the constraints at time $t=0$ (where $m(0)=m_0$) and $t=T$ (where $u(T,x)=G(x,m(T))$).

We are only able to obtain an upper bound for ${\mathcal C}( u, m)-{\mathcal C}^*$ under  additional assumptions: First we assume that $H$ has a separate form: $H=H_0(x,p)-F(x,m)$; Second, we suppose that $\hat {\mathcal G}$ (defined by \eqref{defMathcalFGIntro}) and $\hat {\mathcal F}$  (defined in  a similar way)  are convex (in which case the solution of the MFG system \eqref{eq:MFGcnintro} is unique, see \cite{lasry2007mean}). Then, in Theorem  \ref{thm:UB}, we  show the upper bound:
\begin{align*}
{\mathcal C}( u, m)-{\mathcal C}^*\leq & \;\;
 C \Bigl(\int_0^T \inte \left[ \inte  \frac{\delta F}{\delta m}(x ,y,m(t)) m(t,x)dx\right]^2  dydt \\
&\qquad \qquad\qquad\qquad  +   \inte \left[ \inte  \frac{\delta G}{\delta m}(x, m(T),y) m(T,x)dx\right]^2  dy \Bigr)^{1/2},
\end{align*}
where the constant $C$ depends on the regularity of $H$, $F$ and  $m_0$. As $\delta L/\delta m=\delta F/\delta m$ in the separate case, this lower bound is close to the upper bound given above (with a different exponent, though). We can conclude that, in this setting, the  size of the quantity $ \left\|\int_{\T^d} \frac{\delta F}{\delta m}(y,m,\cdot)m(dy)\right\|_{L^2}$ and $ \left\|\int_{\T^d} \frac{\delta G}{\delta m}(y,m,\cdot)m(dy)\right\|_{L^2}$ along the MFG equilibrium $(u,m)$ controls the difference  ${\mathcal C}(u,m)-{\mathcal C}^*$.

\bigskip

{\it Examples.} To fix the ideas, we assume  that the MFG system is separated: $H=H_0(x,p)-F(x,m)$ and has zero terminal condition: $G\equiv 0$.  We explain in Section \ref{sec.examples} through several examples, that our estimates roughly imply that MFG Nash equilibria are in general inefficient, at least unless the coupling has a very specific structure. 

On the positive side, we  prove  the existence of MFG systems which are {\it globally  efficient}, i.e., such that, for any initial condition $(t_0,m_0)$ there exists a MFG equilibrium $(u,m)$ starting from $(t_0,m_0)$ with ${\mathcal C}(u,m)={\mathcal C}^*$: More precisely,  we show in Theorem \ref{thm.eff} that 
a MFG system is globally  efficient if and only if 
$$
\inte \frac{\delta F}{\delta m}(y,m,x)m(dx)=0,  \qquad \forall (x,m)\in \T^d\times \Pk, 
$$
or, equivalently, if and only if one can write the coupling function $F$ in the form 
$$
F(x,m)= {\mathcal F}(m)+ \frac{\delta {\mathcal F}}{\delta m} (m,x), 
$$
for some map ${\mathcal F}={\mathcal F}(m)$. 
Moreover, one can check (Example \ref{ex0}) that such a coupling function $F$ genuinely depends on $m$ unless $ {\mathcal F}$ is affine.
However, the above structure on $F$ is seldom encountered in practice, and in general there exist (many) initial conditions for which there is no efficient MFG equilibria. 

This is the case for instance if $F=F(m)$ does not depend on $x$ or if  $F$ derives from a potential. In these two cases, the MFG system is globally  efficient  if and only if $F$ is constant (Examples \ref{ex1} and \ref{ex3}). Moreover, our bounds can be simplified in this setting: When $F$ does not depend on $x$, the lower bound for a MFG equilibrium can be rewritten in term of the Holder constant of the map $t\to F(m(t))$: 
$$
{\mathcal C}( u, m)-{\mathcal C}^*\geq 
C_\ep^{-1}\left\{ \sup_{t_1\neq t_2} \frac{\left| F(m(t_2))-F(m(t_1))\right|}{(t_2-t_1)^{1/2}} \right\}^4,
$$
where the supremum is taken over $t_1, t_2\in [\ep, T-\ep]$. 
When $F$ is potential (and thus, as explained in Example \ref{ex3}, ${\mathcal F}$ vanishes and thus is convex), the two inequalities can directly be  expressed  in function of $F$: 
\begin{align*}
C_\ep^{-1}\left(\int_\ep^{T-\ep} \inte \left[   F(y,m(t))\right]^2  dydt\right)^2  \leq \; {\mathcal C}( u, m)-{\mathcal C}^*
\leq C\left(\int_0^T \inte \left[   F(y,m(t)) \right]^2  dydt\right)^{1/2}.
\end{align*}

In the same way, one can show (Example \ref{ex2}) that the MFG equilibria associated with a coupling function of the form  $\ds \ F(x,m)=\int_{\T^d} \phi(x,y)m(dy) \ $,
for some smooth map $\phi:\T^d\times \T^d\to \R$, cannot be globally  efficient unless $\phi$ does not depend on $y$ (and therefore $F$ does not depend on $m$). \\

{\it Extension and limits.} Although we won't make it explicit, one can check that our results   generalize  to other MFG systems (for instance with
local coupling functions or to ergodic MFG systems). However we leave several questions unanswered. First we do not know if the upper bound also holds without our additional assumption. Our technique of proof does not seem to give much result in full generality or requires very restrictive assumptions (see Remark \ref{rem.rem1}). Second, our lower bound seems difficult to generalize to problems with more complex dynamics or for problems with  bounded controls: Indeed our approach strongly relies on the fact that the minimization problem for ${\mathcal C}^*$ has regular solutions, and this requires some assumptions. Finally let us strongly underline that our estimates have little to do with the universal estimates obtained in the context of the ``price of anarchy": Our bounds heavily depend on the regularity of the data and only show how the difference ${\mathcal C}(u,m)-{\mathcal C}^*$ is small or large in function of the specific quantities $ \left\|\int_{\T^d} \frac{\delta F}{\delta m}(y,m,\cdot)m(dy)\right\|_{L^2}$and $ \left\|\int_{\T^d} \frac{\delta G}{\delta m}(y,m,\cdot)m(dy)\right\|_{L^2}$. \\

The paper is organized as follows: In Section \ref{sec:2} we explain our main notations, state our standing assumptions and characterize the minimizers for ${\mathcal C}^*$ in terms of equation \eqref{eq:MFGcnintro}. Section \ref{sec.3} states necessary conditions and sufficient conditions for a MFG equilibrium to be efficient. In Section \ref{sec.4} and \ref{sec.5}, we quantify how far a MFG equilibrium is  from efficiency: Section \ref{sec.5} gives a lower bound and Section \ref{sec.5} an upper bound. We conclude by Section \ref{sec.examples} with the discussion on several examples. \\

{\bf Acknowledgement:} The authors were partially supported by the ANR (Agence Nationale de la Recherche) project ANR-16-CE40-0015-01.

\section{Assumptions and preliminary results} \label{sec:2}

\subsection{Notations and assumptions}

Throughout the paper we work with maps which are all periodic in space, or, in other words, on the $d-$dimensional torus $\T^d:=\R^d/\Z^d$: this simplifying assumption allows us to ignore problems related to  boundary issues or growth conditions of the data.  We denote by $\Pk$ the set of Borel probability measures on $\T^d$, endowed with the Monge-Kantorovitch distance $\dk$: 
$$
\dk(m,m')=\sup_{\phi} \inte \phi(m-m'),
$$
where the supremum is taken over all $1-$Lipschitz continuous maps $\phi:\T^d\to \R$. 

We will use the notion of derivative of a map $U:\Pk\to \R$ as introduced in \cite{cardaliaguet2015master}. We say that $U$ is $C^1$ if there exists a continuous map $\frac{\delta U}{\delta m}:\T^d\times \Pk\to \R$ such that 
$$
U(m')-U(m)=\int_0^1\inte \frac{\delta U}{\delta m}(x,  (1-t)m+tm')(m'-m)(dx) dt \qquad \forall m,m'\in \Pk.
$$
The above relation defines the map $\frac{\delta U}{\delta m}$ only up to a constant. {\it We always use the normalization convention}
\be\label{eq:convention}
\inte \frac{\delta U}{\delta m}(x, m)dm(x)=0 \qquad \forall m\in \Pk.
\ee

If $u:\T^d\times [0,T]\to \R$ is a sufficiently smooth map, we denote by $Du(x,t)$ and $\Delta u(x,t)$ its spatial gradient and spatial Laplacian and by $\partial_t u(x,t)$ its partial derivative with respect to the time variable. 
\bigskip

\noindent {\bf Assumptions.}  The following assumptions are in force throughout the paper. 

\begin{itemize}
\item The Lagrangian $L=L(x,\alpha,m):\T^d\times \R^d\times \Pk \to \R$
is of class $C^2$ with respect to all variables and satisfies 
\be\label{hyp.CvL}
C^{-1}I_d \leq D^2_{pp}L(x,\alpha,m)\leq CI_d.
\ee
We also suppose that 
\be\label{hyp.growth1}
\left|\frac{\delta L}{\delta m} (x,p,m,y)\right|+ \left|\frac{\delta^2 L}{\delta m^2} (x,p,m,y,z)\right|\leq C(1+|p|^2), 
\ee
\be\label{hyp.growth2}
\left|D_\alpha \frac{\delta L}{\delta m} (x,p,m,y)\right|\leq C(1+|p|), \qquad
\left|D^2_{\alpha}L(x,p,m)\right|\leq C. 
\ee
We define the convex conjugate $H$ of $L$ as
$$
L(x,p,m)= \sup_{\alpha\in \R^d} \{ -p\cdot \alpha-L(x,\alpha,m)  \},
$$
and we assume that $H$ is of class $C^2$ as well. Note that $H$ also satisfies: 
\be\label{cond:Hcoercive}
C^{-1}I_d \leq D^2_{pp}H(x,p,m)\leq CI_d.
\ee

\item The coupling function $G: \T^d\times \Pk\to \R$ is globally Lipschitz continuous with space derivatives $\partial_{x_i} G: \T^d\times \Pk\to \R$ also Lipschitz continuous. We also assume that the map $G$ is $C^2$ with respect to $m$ and that its  derivatives $\frac{\delta G}{\delta m}:\T^d\times \Pk\times \T^d\to \R$ and $\frac{\delta^2 G}{\delta m^2}:\T^d\times \Pk\times \T^d\times \T^d\to \R$ are Lipschitz continuous. 
\end{itemize}

We will say below that a constant depends on the regularity of the data if it depends on the horizon $T$, dimension  $d$, on the $C^2$ regularity of $H$, on the constant $C$ in \eqref{cond:Hcoercive}, on the bound on  $G$ and on the modulus of Lipschitz continuity of $\delta G/\delta m$ and of $\delta^2G/\delta m^2$.\\

It will be convenient to set 
\be\label{eq.defmathcalFG}
\hat {\mathcal G}(m):= \inte G(x,m)m(dx), \; \qquad \forall m\in \Pk. 
\ee
Let us compute, for later use, $\delta \widehat{{\mathcal G}}/\delta m$: 

\begin{Lemma}\label{lem:calcul:DF} We have 
\be\label{calcul:DF}
\frac{\delta \widehat{{\mathcal G}}}{\delta m}(m,y)= \int_{\T^d} \frac{\delta G}{\delta m}(x,m,y)m(dx) + G(y,m)-\inte G(x,m)m(dx). 
\ee
\end{Lemma}

\begin{proof} For $m'\in \Pk$, we have 
\begin{align*}
& \lim_{s\to 0^+} \frac{1}{s}(\widehat{{\mathcal G}}((1-s)m+sm')-\widehat{{\mathcal G}}(m)) \\
& \qquad = \lim_{s\to 0^+} \frac{1}{s}\left\{ \int_{\T^d} G(x,((1-s)m+sm')(((1-s)m+sm')(dx)-\int_{\T^d} G(x,m)m(dx)\right\}\\
& \qquad = \lim_{s\to 0^+} \int_{\T^d}  \frac{1}{s}\left\{G(x,((1-s)m+sm')-G(x,m)\right\} m(dx) +\int_{\T^d} G(x,m)(m'-m)(dx) \\
& \qquad =  \int_{\T^d\times \T^d}  \frac{\delta G}{\delta m}(x,m,y)(m'-m)(dy) m(dx) +\int_{\T^d} G(x,m)(m'-m)(dx).
\end{align*}
This implies the claim in view of Convention \eqref{eq:convention}. 
\end{proof}

\subsection{An optimality condition}

We now investigate optimality conditions for the problem, written in an unformal way as 
$$
\min_{(m,w)} \int_{t_0}^T\inte L(x,\frac{w}{m}(t,x),m(t))m(t,x) dxdt + \hat {\mathcal G}(m(T)). 
$$
under the constraint 
$$
\partial_tm -\Delta m +\dive(w)= 0\; {\rm in}\ (t_0,T)\times \T^d, \qquad m(t_0)=m_0\; {\rm in}\ \T^d .
$$
We recall how to give a rigorous meaning to the following expression. We denote by ${\mathcal E}(t_0)$ the  set of  time-dependent Borel measures $(m(t),w(t))\in \Pk\times {\mathcal M}(\T^d,\R^d)$ such that $t\to m(t)$ is continuous,
$$
\int_{t_0}^T |w(t)|dt <\infty, 
$$
and equation
$$
 \partial_t m-\Delta m +\dive(w)= 0 \; {\rm in }\;  [t_0,T]\times \T^d, \:  \: m(t_0)=m_0
 $$
holds in the sense of distribution. We also denote by ${\mathcal E}_2(t_0)$ the subset of $(m(t),w(t))\in {\mathcal E}(t_0)$ such that $w(t)$ is absolutely continuous with respect to $m(t)$ with a density $\frac{dw(t)}{dm(t)}$ satisfying 
$$
\inte  \int_{t_0}^T \left| \frac{dw(t)}{dm(t)}(x)\right|^2 m(dx,t)dt<\infty.
$$
Then we define $J$ on ${\mathcal E}(t_0)$ by 
$$
J(m,w):=
\left\{
\begin{array}{l}
\displaystyle  \int_{t_0}^T\inte L\left(x, \frac{dw(t)}{dm(t)}(x),m(t)\right)m(dx,t) dt+ \hat {\mathcal G}(m(T)) \\
\qquad \qquad  \qquad \qquad \qquad \qquad \qquad \qquad \qquad \qquad \qquad \qquad   \mbox{ if  $(m,w)\in {\mathcal E}_2(t_0)$,}   \\
  \\
   +\infty \qquad \qquad \qquad \qquad \qquad \qquad  \mbox{ otherwise.} 
\end{array}
\right.
$$
We now explain that minimizers of the functional $J$ correspond to solutions of the MFG system. This remark was first pointed out in \cite{lasry2007mean}  and frequently used  since then in different contexts.

\begin{Lemma}\label{lem.OC} Under our standing assumptions, the above problem has at least one solution. Moreover, for any solution $(\hat m,\hat w)$, there exists $\hat u$ such that the pair $(\hat u,\hat m)$  is a classical solution to
\be\label{MFGSocMin}
\left\{ \begin{array}{l}
\ds -\partial_t \hat u -\Delta \hat u+H(x,D\hat u,\hat m(t))= \int_{\R^d} \frac{\delta L}{\delta m}(y,\hat \alpha(t,y), x, \hat m(t))\hat m(t,y)dy 
\; {\rm in}\; (0,T)\times \T^d\\
\ds \partial_t \hat m -\Delta \hat m- \dive(\hat mD_pH(x,D\hat u(t,x),\hat m(t)))=0 \qquad {\rm in}\; (0,T)\times \T^d\\
\ds \hat m(0,x)=m_0(x), \; \hat u(T,x)= \frac{\delta \widehat{{\mathcal G}}}{\delta m}(\hat m(T),x)\qquad {\rm in}\; \T^d\\
\ds \hat \alpha(t,x)= D_pH(x,D \hat u(t,x), \hat m(t))  \qquad {\rm in}\; (0,T)\times \T^d,
\end{array}\right.
\ee
where $H(x,p,m)= \sup_{\alpha\in \R^d}( - \alpha\cdot p- L(x,\alpha,m))$ and 
\be\label{repw}
\hat w(t,x)= -\hat m(t,x) D_pH(x, D\hat u(t,x),\hat m(t)).
\ee
\end{Lemma}

As it has been often pointed out (see  \cite{bensoussan2013mean, carmona2017probabilistic, carmona2015forward, carmona2013control,  huang2012social} for instance), the above system does {\it not} correspond to a mean field game in general because of the extra term on the right-hand side of the Hamilton-Jacobi equation. 

The proof of Lemma \ref{MFGSocMin} is standard and has been described in \cite{briani2018stable} when $H=H_0(x,p)-F(x,m)$ has a separate form. We only explain  the main changes. 

\begin{proof} The existence of a solution can be established exactly as in \cite{briani2018stable}.  
Let now $(\hat  m,\hat  w)$ be a minimum of $J$. For any $(m,w) \in {\mathcal E}$ and $\lambda \in (0,1)$, we set 
$m_\lambda:=(1-\lambda)\hat {m}+\lambda m$, $ w_\lambda:=(1-\lambda)\hat {w}+\lambda w$. We have by minimality of $(\hat  m,\hat  w)$:
\be\label{alizrnedk}
\int_0^T \inte L(x, \frac{w_\lambda}{m_\lambda}, m_\lambda) m_\lambda +{\hat {\mathcal G}}(m_\lambda (T))  
\geq \int_0^T \inte L(x, \frac{\hat  w}{\hat  m}, \hat  m) \hat  m +{\hat {\mathcal G}}(\hat  m (T)) .
\ee
By the convexity condition of $L=L(x,\alpha,m)$ in \eqref{hyp.CvL}, the map $(s,z)\to L(x,s/z,m)s$ on $(0,+\infty)\times \R^d$ is convex for any fixed $x,m$. So we have 
$$
\int_0^T \inte L(x, \frac{w_\lambda}{m_\lambda}, m_\lambda) m_\lambda
\leq 
(1-\lambda)  \int_0^T \inte L(x, \frac{\hat  w}{\hat  m}, m_\lambda) \hat  m+
\lambda  \int_0^T \inte L(x, \frac{w}{m}, m_\lambda) m. 
$$
So we can rewrite \eqref{alizrnedk} as: 
\begin{align*}
&\lambda \left(\int_0^T \inte L(x, \frac{w}{m}, m_\lambda) m -  \int_0^T  \inte L(x, \frac{\hat  w}{\hat  m}, m_\lambda) \hat  m \right) \\
&\qquad \geq  \int_0^T \inte L(x, \frac{w_\lambda}{m_\lambda}, m_\lambda) m_\lambda - \int_0^T \inte L(x, \frac{\hat  w}{\hat  m}, m_\lambda) \hat  m \\
&\qquad \geq  -\left(  \int_0^T \inte L(x, \frac{\hat  w}{\hat  m}, m_\lambda) \hat  m - \int_0^T \inte L(x, \frac{\hat  w}{\hat  m}, \hat  m) \hat  m \right)-({\hat {\mathcal G}}(m_\lambda (T)) -{\hat {\mathcal G}}(\hat  m (T))) .
\end{align*}
Thus dividing by $\lambda>0$ and letting $\lambda \to 0^+$ we find, thanks to the regularity of $L$ and ${\hat {\mathcal G}}$: 
\begin{align*}
&\int_0^T \inte L(x, \frac{w(t,x)}{m(t,x)}, \hat  m(t)) \hat m(t,x)dxdt -  \int_0^T  \inte L(x, \frac{\hat  w(t,x)}{\hat  m(t,x)}, \hat  m(t)) \hat  m(t,x)dxdt  \\
&\qquad \geq  - \int_0^T \inte\inte \frac{\delta L}{\delta m} (x, \frac{\hat  w(t,x)}{\hat  m(t,x)}, y, \hat  m) \hat  m(t,x)(m(t,y)-\hat  m(t,y))dxdydt  \\
&\qquad  \qquad -\inte \frac{\delta {\hat {\mathcal G}}}{\delta m}(\hat  m(T),y)(m(T,y)-\hat  m (T,y))dy .
\end{align*}
This means that the pair $(\hat  m,\hat  w)$ is optimal for the (local) functional 
\begin{align*}
{\mathcal J}(m,w)& := 
\int_0^T \inte L(x, \frac{w(t,x)}{m(t,x)}, \hat  m(t)) m(t,x)dxdt  \\
&\qquad + \int_0^T \inte\inte \frac{\delta L}{\delta m} (x, \frac{\hat  w(t,x)}{\hat  m(t,x)}, y, \hat  m(t)) \hat  m(t,x)m(t,y)dxdydt  \\
&\qquad + \inte \frac{\delta {\hat {\mathcal G}}}{\delta m}(\hat  m(T),y)m(T,y)dy.
\end{align*}
We can then conclude exactly as in \cite{briani2018stable} that there exists $\hat  u$ such that $(\hat  u,\hat m)$  is a classical solution to the MFG system
\eqref{MFGSocMin} and that $\hat w=-\hat m D_pH(x, D\hat u,\hat m(t))$. 
 \end{proof}


\section{Efficiency of MFG equilibria}\label{sec.3}

Let $(t_0,m_0)\in [0,T]\times \Pk$ be an initial distribution and $(u,m)$ be the solution of the MFG system
\be\label{MFGt}
\left\{ \begin{array}{l}
-\partial_t u -\Delta u+H(x,Du, m(t))=0\qquad {\rm in}\; (t_0,T)\times \T^d\\
\partial_t m -\Delta m-\dive(mD_pH(x,Du,m(t)))=0\qquad {\rm in}\; (t_0,T)\times \T^d\\
m(t_0,x)=m_0(x), \; u(T,x)= G(x,m(T))\qquad {\rm in}\;  \T^d.
\end{array}\right.
\ee
The social cost associated with the equilibrium $(u,m)$ is defined by
$$
{\mathcal C}(u,m)= \int_{t_0}^T\int_{\T^d} \left\{ L(x,\alpha^*(t,x))+ F(x,m(t))\right\} m(t,x)\ dxdt+ \inte G(x,m(T))m(T,x)dx,
$$
where $\alpha^*(t,x):=-D_pH(x,Du(t,x))$ and $L(x,\alpha,m)=\sup_{p\in \R^d} (-\alpha\cdot p- H(x,p,m))$. 

We want to compare ${\mathcal C}(u,m)$
 with the social cost obtained by a global planner, which is defined as 
\be\label{SocMin}
{\mathcal C}^*:= \inf_{(m,\alpha)} \int_{t_0}^T\int_{\T^d} \left\{ L(x,\alpha(t,x),m(t))\right\} m(t,x)\  dxdt+ \inte G(x,m(T))m(T,x)dx,
\ee
where the infimum is taken over the pairs $(m,\alpha)$ such that 
\be\label{eq:ConstrT}
\partial_t m -\Delta m+\dive(m\alpha)=0,\;  {\rm in}\; (t_0,T)\times \T^d, \qquad
m(t_0,x)=m_0(x)\qquad  {\rm in}\;  \T^d. 
\ee
Although ${\mathcal C}^*$ depends on the initial position $(t_0,m_0)$, we will omit to write this dependence explicitly  to simplify the expressions. 

We say that an equilibrium $(u,m)$, solution of the MFG system \eqref{MFGt}, is {\it efficient} if 
$$
{\mathcal C}(u,m)={\mathcal C}^*.
$$
We say that the MFG system \eqref{MFGt} is {\it globally efficient} if,  for any initial position $(t_0,m_0)\in [0,T]\times \Pk$, there exists an efficient MFG equilibrium with initial position $(t_0,m_0)$. 

\subsection{A necessary condition for efficiency}

\begin{Proposition}\label{prop.eff} Let $(u,m)$ be a MFG equilibrium, i.e., a solution to \eqref{MFGt}. If $(u,m)$ is efficient, then, for any  $(t,x)\in [0,T]\times \T^d$,
$$
 \int_{\T^d} \frac{\delta L}{\delta m}(y,\alpha^*(t,y),x, m(t))m(t,y)dy=0\qquad {\rm and}\qquad \int_{\T^d} \frac{\delta G}{\delta m}(t,m(T),x)m(T,y)dy=0,
 $$
 where $\alpha^*(t,x):= -D_pH(x,Du(t,x),m(t))$. 
\end{Proposition}

\begin{proof} Assume that equality ${\mathcal C}(u,m)= {\mathcal C}^*$ holds. Then the pair $(m, \alpha^*)$ is a minimizer for ${\mathcal C}^*$. By the  characterization of minimizers in Lemma \ref{lem.OC}, there exists $v$ such that the pair $(v,m)$ solves system \eqref{MFGSocMin} with, by \eqref{repw},  
$$
\alpha^*(t,x)=-D_pH(x,Du(t,x),m(t))= -D_pH(x, Dv(t,x),m(t))\qquad \forall (t,x)\in [t_0,T]\times \T^d. 
$$ 
By injectivity of $D_pH$ with respect to the second variable (coming from the strict convexity of $H$ with respect to $p$), we get $Du=Dv$.  This implies that there is a constant $c(t)$ such that $u(t,x)=v(t,x)+c(t)$. By the equations satisfied by $u$ and $v$ we have therefore, for any $(t,x)$, 
\be\label{aelrbzrnejd}
\ds - c'(t)=\int_{\R^d} \frac{\delta L}{\delta m}(y, \alpha^*(t,y), x,  m(t)) m(t,y)dy
.
\ee
We integrate the above equality against $m(t)$: 
$$
\ds -c'(t)= \int_{\T^d}\int_{\R^d} \frac{\delta L}{\delta m}(y, \alpha^*(t,y), x,  m(t)) m(t,y)m(t,x) dydx .
$$
As the double integral vanishes because of Convention \eqref{eq:convention}, we get $c'(t)=0$, 
and therefore, coming back to \eqref{aelrbzrnejd}, 
$$
\int_{\R^d} \frac{\delta L}{\delta m}(y, \alpha^*(t,y), x,  m(t)) m(t,y)dy =0\qquad \forall (t,x)\in [0,T]\times \T^d.
 $$
Equality $u(T,x)=v(T,x)+c(T)$ also implies by Lemma \ref{lem:calcul:DF} that 
\begin{align*}
& G(x,m(T))  = \frac{\delta \widehat{{\mathcal G}}}{\delta m}(m(T), x)+c(T)\\
& \qquad = \int_{\T^d} \frac{\delta G}{\delta m}(y,m(T),x)m(T,y)dy + G(x,m(T))-\inte G(y,m(T))m(T,y)dy+c(T). 
\end{align*}
Integrating with respect to $m(T,x)dx$ and using Convention \eqref{eq:convention}, we obtain: 
$$
0=-\inte G(y,m(T))m(T,y)dy+c(T),
$$
and therefore 
$$
 \int_{\T^d} \frac{\delta G}{\delta m}(y,m(T),x)m(T,y)dy=0\qquad \forall x\in \T^d.
 $$
\end{proof}

\subsection{Characterization of the global efficiency}

Let us recall that we say that the MFG system \eqref{MFGt} is {\it globally efficient} if,  for any initial position $(t_0,m_0)\in [0,T]\times \Pk$, there exists an efficient MFG equilibrium with initial position $(t_0,m_0)$. In order to proceed and characterize global efficiency, we need to work in a special case: we assume that $H$ has the separate form 
\be\label{eq.separate}
H(x,p,m) = H_0(x,p)-F(x,m)\qquad \forall (x,p,m)\in \T^d\times \R^d\times \Pk.
\ee
Then $L$ is also in a separate form: 
$$
L(x,\alpha,m)= L_0(x,\alpha)+F(x,m),
$$
where $L_0(x,\alpha)= \sup_p \alpha\cdot p -H_0(x,p)$ is the convex conjugate of $H_0$ with respect to the last variable.  In this case, the MFG system \eqref{MFGt} becomes: 
\be\label{MFGtSep}
\left\{ \begin{array}{l}
-\partial_t u -\Delta u+H_0(x,Du)-F(x, m(t))=0\qquad {\rm in}\; (t_0,T)\times \T^d\\
\partial_t m -\Delta m-\dive(mD_pH(x,Du))=0\qquad {\rm in}\; (t_0,T)\times \T^d\\
m(t_0,x)=m_0(x), \; u(T,x)= G(x,m(T))\qquad {\rm in}\;  \T^d.
\end{array}\right.
\ee
Note also that $\delta L/\delta m$ reduces to 
$$
\frac{\delta L}{\delta m}(y,\alpha,x, m)= \frac{\delta F}{\delta m}(y, x, m), 
$$
where the right-hand side is independent of $\alpha$. In this case, Proposition \ref{prop.eff} states that, if the MFG equilibrium $(u,m)$ is efficient, then, for any $(t,x)\in [t_0,T]\times \T^d$,   
$$
 \int_{\T^d} \frac{\delta F}{\delta m}(y,x,m(t))m(t,y)dy=0\qquad {\rm and}\qquad \int_{\T^d} \frac{\delta G}{\delta m}(y,x,m(T))m(T,y)dy=0.
 $$

The following statement is a kind of converse. 

\begin{Proposition}\label{prop.z} Assume  that $H$ is of separate form (i.e., \eqref{eq.separate} holds) and that, for any $(x,m)\in \T^d\times \Pk$,  
\be\label{eq:condefficient}
\inte \frac{\delta F}{\delta m}(y, m,x) m(dy)= 0\;{\rm and}\; \inte \frac{\delta G}{\delta m}(y, m,x) m(dy)= 0. 
\ee
Then, the MFG system is globally  efficient: for any initial condition $(t_0,m_0)\in [0,T]\times \Pk$, there exists a solution $(u,m)$ to the MFG system \eqref{MFGtSep} such that 
$$
{\mathcal C}(u,m)= {\mathcal C}^*.
$$
\end{Proposition}

\begin{proof} Without loss of generality, we can  assume that $m_0$ has a smooth and positive density. Otherwise we can proceed by approximation. Let $(\hat m, \hat \alpha)$ be the minimum of \eqref{SocMin} and $\hat u$ be such that $(\hat u, \hat m)$ solves \eqref{MFGSocMin} (recall Lemma \ref{lem.OC}).  By assumption \eqref{eq:condefficient} and the structure condition \eqref{eq.separate}, $(\hat u, \hat m)$ solves 
$$
\left\{ \begin{array}{l}
\ds -\partial_t \hat u -\Delta \hat u+H_0(x,D\hat u)- F(x,\hat m(t)) =0\qquad {\rm in}\; (t_0,T)\times \R^d\\
\ds \partial_t \hat m -\Delta \hat m- \dive(\hat mD_pH(x,D\hat u(t,x)))=0 \qquad {\rm in}\; (t_0,T)\times \R^d\\
\ds \hat m(t_0,x)=m_0(x), \; \hat u(T,x)= \frac{\delta \widehat{{\mathcal G}}}{\delta m}(\hat m(T),x)\qquad {\rm in}\; \R^d.
\end{array}\right.
$$
As, by  Lemma \ref{lem:calcul:DF}, 
$$
\frac{\delta \widehat{{\mathcal G}}}{\delta m}(m,x)= G(x,m)  -\inte G(y,m)m(dy)
$$
we see that $u(t,x)= \hat u(t,x)+\inte G(y,m(T))m(dy)$ solves the MFG system \eqref{MFGtSep}. 
Moreover, by the definition of $u$, $\hat u$ and $\hat m$, 
 $$
 {\mathcal C}(u,\hat m) = {\mathcal C}(\hat u, \hat m) = {\mathcal C}^*. 
 $$
\end{proof}
 
Let us now point out an equivalent form of \eqref{eq:condefficient}: 

\begin{Proposition}\label{prop:condefficientEQ} The map $F$ satisfies \eqref{eq:condefficient} if and only if there exists a $C^2$ function ${\mathcal F}:\Pk\to \R$ such that
\be\label{eq:condefficientEQ}
F(x,m)= {\mathcal F}(m)+ \frac{\delta {\mathcal F}}{\delta m} (m,x)\qquad \forall (x,m)\in \T^d\times \Pk.
\ee
In this case, one can take ${\mathcal F}= \hat {\mathcal F}$, where $\hat {\mathcal F}$ is given by
$$
\hat {\mathcal F}(m):=\inte F(x,m)m(dx)\qquad \forall m\in \Pk.
$$ 
\end{Proposition}

\begin{proof} If \eqref{eq:condefficient} holds, then it is obvious by  \eqref{calcul:DF} that $F$ is of the form \eqref{eq:condefficientEQ} with ${\mathcal F}= \hat {\mathcal F}$. Conversely, if $F$ is of the form \eqref{eq:condefficientEQ}, then 
\begin{align*}
\frac{\delta F}{\delta m}(x,m,y) & = \frac{\delta {\mathcal F}}{\delta m}(m,y)+ \frac{\delta^2 {\mathcal F}}{\delta m^2} (m,x,y) 
= \frac{\delta^2 {\mathcal F}}{\delta m^2} (m,y,x)+ \frac{\delta {\mathcal F}}{\delta m}(m,x), 
\end{align*}
because, according to \cite[Lemma 2.2.4]{cardaliaguet2015master},  
$$
\frac{\delta^2 {\mathcal F}}{\delta m^2} (m,x,y)= \frac{\delta^2 {\mathcal F}}{\delta m^2} (m,y,x)-\frac{\delta {\mathcal F}}{\delta m}(m,y)+ \frac{\delta {\mathcal F}}{\delta m}(m,x).
$$
We can then conclude that 
$$
\inte \frac{\delta F}{\delta m}(x,m,y)m(dx)= \inte \frac{\delta^2 {\mathcal F}}{\delta m^2} (m,y,x)m(dx)+ \inte \frac{\delta {\mathcal F}}{\delta m}(m,x)m(dx) =0, 
$$ 
by Convention \eqref{eq:convention}. 
\end{proof}

To summarize, we have obtained the following characterization of global efficiency: 
\begin{Theorem}\label{thm.eff}  Assume  that $H$ is of separate form (i.e., \eqref{eq.separate} holds). Then the MFG system is globally efficient if and only if 
$$
\inte \frac{\delta F}{\delta m}(y,m,x)m(dx)=0, \; \inte \frac{\delta G}{\delta m}(y,m,x)m(dx)=0, \qquad \forall (x,m)\in \T^d\times \Pk, 
$$
which is also equivalent to the existence of $C^2$ maps ${\mathcal F}:\Pk\to \R$ and ${\mathcal G}:\Pk\to \R$ such that 
$$
F(x,m)= {\mathcal F}(m)+ \frac{\delta {\mathcal F}}{\delta m} (m,x),\; G(x,m)= {\mathcal G}(m)+ \frac{\delta {\mathcal G}}{\delta m} (m,x)
\qquad \forall (x,m)\in \T^d\times \Pk.
$$
\end{Theorem}

Surprisingly, this condition depends only on the coupling terms $F$ and $G$, but not on the Hamiltonian $H_0$. 

\begin{proof} We have seen in Propositions \ref{prop.z} and \ref{prop:condefficientEQ} that the existence of ${\mathcal F}$ for which \eqref{eq:condefficientEQ} holds  is sufficient for the global efficiency. Conversely, if the MFG system is globally efficient, then, for any initial condition $(t_0,m_0)\in [0,T]\times \Pk$, where $m_0$ has a smooth density, there exists a MFG equilibrium $(u,m)$ such that  
$$
 \int_{\T^d} \frac{\delta F}{\delta m}(y,x,m(t))m(t,y)dy=0\qquad {\rm and}\qquad \int_{\T^d} \frac{\delta G}{\delta m}(y,x,m(T))m(T,y)dy=0.
 $$
In particular, for $t=t_0$, we obtain
$$
 \int_{\T^d} \frac{\delta F}{\delta m}(y,x,m_0)m_0(y)dy=0.
 $$
 Choosing $t_0$ arbitrarily close to $T$, $m(T)$ becomes closer and closer to $m_0$ (because $m_0$ has a smooth density), we obtain: 
 $$
 \int_{\T^d} \frac{\delta G}{\delta m}(y,x,m_0)m_0(y)dy=0.
 $$
 We conclude by approximation and using Proposition \ref{prop:condefficientEQ} again. 
\end{proof}

\section{Lower bound on ${\mathcal C}-{\mathcal C}^*$}\label{sec.4}

\begin{Theorem}\label{thm:LB} Under our standing assumptions, let  $(u,m)$ be a solution to the MFG system \eqref{MFGt} starting from $(t_0,m_0)\in [0,T]\times \Pk$. Then we have the lower bound: for any $\ep>0$, 
 \begin{align}
{\mathcal C}( u, m)-{\mathcal C}^*\geq & \; \;
 C_\ep^{-1}\Bigl(\int_{t_0+\ep}^{T-\ep} \inte \left[ \inte  \frac{\delta L}{\delta m}(x, \alpha^*(t,x),y,m(t)) m(t,x)dx\right]^2  dydt \Bigr)^2 \label{LBt}\\
&\qquad \qquad\qquad\qquad  +   C^{-1}\Bigl(\inte \left[ \inte  \frac{\delta G}{\delta m}(x, m(T),y) m(T,x)dx\right]^2  dy \Bigr)^4,
\notag
\end{align}
where $\alpha^*(t,x)=-D_pH(x,Du(t,x),m(t))$ and where the constants $C\geq 1$  depends on the regularity of $H$,  $G$ and on $m_0$ and where  $C_\ep\geq 1$ depends also on $\ep$. 
\end{Theorem}

\begin{Remark}{\rm
The presence of $\ep$ is related to the constraints $m(t_0)=m_0$ and $u(T)=G(x,m(T))$: they prevent the choice of arbitrary test functions in the proof. For instance, if $G\equiv 0$, one can replace $T-\ep$ by $T$ in the integral. 
}\end{Remark}

\begin{proof}[Proof of Theorem \ref{thm:UB}] As the equation for $m$ is uniformly parabolic, for $\ep>0$ there exists a constant $C_\ep$ such that  $m$ has  a $C^2$  density which is bounded below by $C_\ep^{-1}$ on $[t_0+\ep/2,T]$.  With a given  $(\mu,\beta)$ smooth solution to 
\be\label{eq.linbetamu}
\partial_t \mu-\Delta \mu+\dive(\beta)=0\qquad {\rm in }\; (t_0,T)\times \T^d,\qquad \mu(t_0,x)=0\; {\rm in} \; \T^d, 
\ee
with $\beta=\mu=0$ on $[t_0,t_0+\ep/2]$, we set $\tau_\ep:= 1/(2C_\ep\|\mu\|_\infty)$ and, for $h\in [0,\tau_\ep]$, $(m_h,\alpha_h):=( m+h \mu, ( m\alpha^*+h\beta)/( m+h \mu))$ where $\alpha^*(t,x):=-D_pH(x,Du(t,x),m(t))$. Note that  the pair $(m_h,\alpha_h)$ satisfies $m_h(t)\in \Pk$ for any $t$ and the constraint \eqref{eq:ConstrT} for any $h\in [0,\tau_\ep]$ (in particular $m_h(t_0)=m_0$ because $\mu(t_0)=0$). Moreover, $h\to (m_h,\alpha_h)$ is smooth because $\mu=0$ on $[t_0,t_0+\ep/2]$ and $m$ is bounded below by a positive constant on $[\ep/2,T]$. 

Next we define the map $\phi:[0,\tau_\ep]\to \R$ by 
\begin{align*}
\phi(h)& :=\int_{t_0}^T \inte L\left(x, \alpha_h(t,x),  m_h(t)\right)  m_h(t,x)dxdt
+ \inte G(x,m_h(T))m_h(T,x)dx.
\end{align*}
 We have 
$$
\begin{array}{rl}
\ds \phi'(h) \; =  & \ds  \int_{t_0}^T \inte L\left(x, \alpha_h(t,x),m_h(t)\right)\mu(t,x) dxdt\\
& \ds + \int_{t_0}^T \inte D_\alpha L(x,  \alpha_h(t,x),m_h(t))\cdot \left(\beta(t,x)- \alpha_h(t,x)\mu(t,x)\right) dxdt\\
& \ds +\int_{t_0}^T  \inte \inte \frac{\delta L}{\delta m}(x, \alpha_h(t,x),y,m_h(t)) \mu(t,y) m_h(t,x)dydxdt\\
& \ds + \inte \left\{ G(x, m_h(T))\mu(T,x)+ \inte \frac{\delta G}{\delta m}(x, m_h(T),y) \mu(T,y) m_h(T,x)dy\right\}dx.
\end{array}
$$
Recalling the definition of $ \alpha^*$ and the fact that 
$$
\begin{array}{rl}
\ds L\left(x, \alpha^*(t,x),m(t)\right)\; =  & \ds -H(x, D u(t,x),m(t))+ D_pH(x,D u(t,x),m(t))\cdot D u(t,x)\\ 
= & \ds -\partial_t u(t,x)-\Delta  u(t,x)-  \alpha^*(t,x)\cdot D u(t,x), 
\end{array}
$$
and that 
$$
D_\alpha L(x,  \alpha^*(t,x),m(t)) = -D u(t,x), 
$$
we obtain 
\be\label{phiprim0}
\begin{array}{rl}
\ds \phi'(0) \; =  & \ds \int_{t_0}^T \inte  (-\partial_t u(t,x)-  \Delta u(t,x) - \alpha^*(t,x)\cdot D u(t,x))  \mu (t,x) dxdt\\
& \ds + \int_{t_0}^T \inte 
-D u(t,x)\cdot \left( \beta(t,x)- \alpha^*(t,x)\mu(t,x)\right) dxdt\\
& \ds +\int_{t_0}^T  \inte \inte \frac{\delta L}{\delta m}(x, \alpha^*(t,x),m(t),y) \mu(t,y) m(t,x)dydxdt\\
& \ds + \inte \left\{ G(x, m(T))\mu(T,x)+ \inte \frac{\delta G}{\delta m}(x, m(T),y) \mu(T,y) m(T,x)dy\right\}dx\\
 = & \ds \int_{t_0}^T  \inte \inte \frac{\delta L}{\delta m}(x, \alpha^*(t,x),m(t),y) \mu(t,y) m(t,x)dydxdt\\
& \ds + \inte \inte \frac{\delta G}{\delta m}(x, m(T),y) \mu(T,y) m(T,x)dydx,
\end{array}
\ee
where we used the equation satisfied by the pair $(\mu,\beta)$ and the fact that $u(T,\cdot)=G(\cdot, m(T))$  for the last equality. Let us also note for later use that 
\be\label{eq.phisecon}
\begin{array}{l}
\ds \phi''(h) 
\ds 
= \int_{t_0}^T\inte  D^2_{\alpha}L\left(x, \alpha_h,m_h\right) (\beta-\mu \alpha_h) \cdot (\frac{\beta}{m_h}-\frac{\mu \alpha_h}{m_h}) dxdt
\\
\qquad \qquad  \ds 
 + 2\int_{t_0}^T\inte\inte D_\alpha \frac{\delta L}{\delta m}(x, \alpha_h,m_h,y)\cdot (\beta-\alpha_h\mu)(t,x)  \mu(t,y)dxdydt
 \\
\qquad \qquad  \ds + \int_{t_0}^T\inte \inte \frac{\delta^2   L}{\delta m^2} (x,\alpha_h,  m_h,y,z)\mu(y,t)\mu(z,t)m_h(t,x)dxdydt\\
\qquad \qquad  \ds  
+2\int_{t_0}^T  \inte \inte \frac{\delta L}{\delta m}(x, \alpha_h(t,x),m_h(t),y) \mu(t,y) \mu(t,x)dydxdt\\
\qquad \qquad  \ds +
2 \inte \inte  \frac{\delta G}{\delta m^2} (x,m_h(T),y)\mu(x,T)\mu(y,T)dxdy\\
\qquad \qquad  \ds +
\inte \inte \frac{\delta^2 G}{\delta m^2} (x,m_h(T),y,z)\mu(x,T)\mu(y,T)m_h(t,x)dxdy.
\end{array}
\ee
Recall that,  for any $h\in [0,\tau_\ep]$, the pair $(m_h,\alpha_h)$ satisfies $m_h(t)\in \Pk$ for all $t\in [t_0,T]$ and the constraint \eqref{eq:ConstrT}. Therefore
$$
\phi(h)\geq {\mathcal C}^*\qquad \forall h\in [0,\tau_\ep].
$$
As 
$$
\phi(h)\leq \phi(0)+ h \phi'(0)+ \frac{h^2}{2}\|\phi''\|_\infty
= {\mathcal C}(u,m)+ h \phi'(0)+ \frac{h^2}{2}\|\phi''\|_\infty, 
$$
we obtain, 
\be\label{boundmathC}
{\mathcal C}(u,m)-{\mathcal C}^* \geq  -h \phi'(0)- \frac{h^2}{2}\|\phi''\|_\infty\qquad \forall h\in [0,\tau_\ep].
\ee

We now apply the above computations to two particular cases, one to get the lower bound involving $F$ and the other one for the lower bound involving $G$.  For $\ep\in (0,(T-t_0)/2)$, let us set: 
\be\label{defmumu}
\mu(t,y):= -\gamma(t)m(t,y) \int_{\T^d} \frac{\delta L}{\delta m}(x, \alpha^*(t,x), y,m(t)) m(t,x)dx, 
\ee
where 
$$
\gamma(t):=\left\{\begin{array}{ll}
1 & {\rm if }\; t\in [t_0+\ep, T-\ep]\\
0  & {\rm if }\; t\in [t_0,t_0+\ep/2]\\
2(t-\ep/2)/\ep & {\rm if }\; t\in [t_0+\ep/2,t_0+\ep]\\
(T-t)/\ep & {\rm if }\; t\in [T-\ep,T]
\end{array}\right.
$$
By Convention \eqref{eq:convention},  
we have $\inte \mu(t,y)dy=0$, so that we can find a continuous map $\beta=\beta(t,x)$ such that \eqref{eq.linbetamu} holds. 
Note that $\mu$ is uniformly bounded by a constant $C_1$ in $L^\infty$ independently of $\ep$ and therefore we can choose $\tau_\ep= 1/(C_0C_1)$. Let us define $\phi$ as above. Then $\phi$ is of class $C^{1,1}$ with $\|\phi''\|_\infty\leq C_2$, where $C_2$ depends only on $\ep$ and on the regularity of $H$ and $F$ and $C_0$.

By \eqref{phiprim0} and the definition of $\mu$ in \eqref{defmumu} we have 
\begin{align*}
\phi'(0) & =   - \int_{t_0}^{T}\gamma(t) \inte \left[ \inte  \frac{\delta L}{\delta m}(x, \alpha^*(t,x),y,m(t)) m(t,x)dx\right]^2  m(t,y)dydt \\
& \leq - \int_{t_0+\ep}^{T-\ep} \inte \left[ \inte  \frac{\delta L}{\delta m}(x, \alpha^*(t,x),y,m(t))  m(t,x)dx\right]^2  m(t,y)dydt. 
\end{align*}
Thus, applying \eqref{boundmathC} with $h=\min\{\tau_\ep, -\phi'(0)/\|\phi''\|_\infty\}$, we obtain our first lower bound: 
$$
{\mathcal C}-{\mathcal C}^* \geq C_\ep^{-1}\int_{t_0+\ep}^{T-\ep} \inte \left[ \inte  \frac{\delta L}{\delta m}(x, \alpha^*(t,x),y,m(t))  m(t,x)dx\right]^2  m(t,y)dydt,
$$
for some constant $C_\ep$ depending on the data, on $m_0$ and on $\ep$. \\

In order to obtain the lower bound involving $G$, we choose
\be\label{defmumu2}
\mu(t,y):=- \gamma(t)m(T,y) \int_{\T^d} \frac{\delta G}{\delta m}(x, m(T), y) m(T,x)dx, 
\ee
where 
$$
\gamma(t):=\left\{\begin{array}{ll}
0 & {\rm if }\; t\in [t_0, T-\ep]\\
(t-(T-\ep))/\ep & {\rm if }\; t\in [T-\ep,T]
\end{array}\right.
$$
where $\ep\in (0,T/2)$ is small. Note that we can choose the lower bound $C_\ep$ such that $m\geq C_\ep^{-1}$ on $[T/2,T]$ independent of $\ep$. Hence the constant $\tau_\ep:= 1/(2C_\ep\|\mu\|_\infty)$ does not depend on $\ep$ either, and we call it $\tau_0$. 

As before we  can find a continuous map $\beta=\beta(t,x)$ such that \eqref{eq.linbetamu} holds. As $\|\gamma'\|_\infty\leq T/\ep$, we have $\|\beta\|_\infty\leq C/\ep$ where $C$ depends on the  regularity of $G$ only. Moreover $\mu$ is bounded in $L^\infty$ and, therefore, $\|\alpha_h\|_\infty\leq C/\ep$.

Let $\phi$ be associated to $(\mu,\beta)$ as above. Then, by \eqref{eq.phisecon} and our growth assumptions \eqref{hyp.growth1}, \eqref{hyp.growth2}, we have $\|\phi''\|_\infty\leq C/\ep^2$. On the other hand, from the choice of $\mu$ and \eqref{phiprim0}, 
\begin{align*}
\phi'(0) & =   \int_{T-\ep}^{T}\gamma(t)   \inte \inte \frac{\delta L}{\delta m}(x, \alpha^*(t,x),m(t),y) \mu(t,y) m(t,x)dydxdt\\
& \qquad \ds + \inte \inte \gamma(T)\frac{\delta G}{\delta m}(x, m(T),y) \mu(T,y) m(T,x)dydx,\\
& \leq  C\ep - \kappa ,
\end{align*}
with
$$
\kappa:= \inte \left[ \inte  \frac{\delta G}{\delta m}(x, m(T),y) m(T,x)dx\right]^2  m(T,y)dy.
$$
We now use \eqref{boundmathC} to obtain
$$
{\mathcal C}(u,m)-{\mathcal C}^* \geq (\kappa -C\ep)h -\frac{C}{2\ep^2}h^2 \qquad \forall h\in [0,\tau_0].
$$
Choosing $\ep= c\kappa$ (for some constant $c>0$ small enough) and $h= \min\{\tau_0, C^{-1}\kappa^3\}$ (for some large constant $C$ with the same dependence as above), we get our second lower bound:
$$
{\mathcal C}(u,m)-{\mathcal C}^* \geq C^{-1} \kappa^4. 
$$
Putting together our two lower bounds on ${\mathcal C}-{\mathcal C}^*$, we finally obtain Inequality \eqref{LBt}.
\end{proof}

\section{Upper bounds on  ${\mathcal C}-{\mathcal C}^*$.} \label{sec.5}

In order to obtain an upper bound for  ${\mathcal C}-{\mathcal C}^*$, we come back to the case where $H$ is separated, i.e., satisfies \eqref{eq.separate}. 
Let us recall that, in this case, the MFG system becomes \eqref{MFGtSep}. Recall the notation 
$$
\hat {\mathcal F}(m):= \inte F(x,m)m(dx), \qquad
\hat {\mathcal G}(m):= \inte G(x,m)m(dx), \; \qquad \forall m\in \Pk.
$$

\begin{Theorem}\label{thm:UB} Under our standing assumptions, assume that  the initial condition $m_0$  has a smooth and positive density. Assume in addition that the maps $\hat {\mathcal F}$ and $\hat {\mathcal G}$ 
are convex  on $\Pk$. Let  $(u,m)$ be a solution to the MFG system \eqref{MFGtSep} starting from $(t_0,m_0)$. Then
we also have the upper bound:
\begin{align}
{\mathcal C}( u, m)-{\mathcal C}^*\leq & \;\;
 C \Bigl(\int_{t_0}^T \inte \left[ \inte  \frac{\delta F}{\delta m}(x ,y,m(t)) m(t,x)dx\right]^2  dydt \label{UBt}\\
&\qquad \qquad\qquad\qquad  +   \inte \left[ \inte  \frac{\delta G}{\delta m}(x, m(T),y) m(T,x)dx\right]^2  dy \Bigr)^{1/2},
\notag
\end{align}
where the constants $C\geq 1$  depends on the regularity of $H_0$, $F$, $G$ and on $m_0$.
\end{Theorem}

\begin{Remark}\label{rem.rem1}{\rm  The result can actually be generalized to   MFG systems with non separated Hamiltonian. However, in this case, the convexity condition has to be stated on the map 
$$
(m,w) \to \inte L(x,w/m,m)dm.
$$
However, as this later condition seems very restrictive, we have chosen to state the result for separated Hamiltonians. 
}\end{Remark}

\begin{proof}[Proof of Theorem \ref{thm:UB}] We compute as usual (see \cite{lasry2007mean})
$$
\frac{d}{dt} \int_{\T^d} (u-\hat u)(m-\hat m).
$$
We have, since $m(0)=\hat m(0)=m_0$,  
$$
\begin{array}{l}
\ds 0 = \int_{t_0}^T\int_{\T^d} m(H_0(x,D\hat u)- H_0(x,Du)-D_pH_0(x,Du)\cdot (D\hat u-Du) dxdt\\
\qquad \ds + \int_{t_0}^T\int_{\T^d} \hat m(H_0(x,Du)- H_0(x,D\hat u)-D_pH_0(x,D\hat u)\cdot (D u-D\hat u)dxdt \\
\qquad \ds + \int_{t_0}^T\int_{\T^d} \left(F(x,m(t))-\frac{\delta \widehat{ {\mathcal F}}}{\delta m}(\hat m(t),x)\right)(m(t,x)-\hat m(t,x))dxdt\\
\qquad \ds  + \int_{\T^d} \left(G(x,m(T))-\frac{\delta \widehat{ {\mathcal G}}}{\delta m}(\hat m(T),x)\right)(m(T,x)-\hat m(T,x))dx.
\end{array}
$$
Using the  uniform convexity of $H_0$, we find: 
\be\label{jlhbfzsd}
\begin{array}{l}
\ds C^{-1}\int_{t_0}^T \int_{\T^d} (m(t,x)+\hat m(t,x))|Du-D\hat u|^2dxdt \\
\ds \qquad \leq - \int_{t_0}^T\int_{\T^d} \left(F(x,m(t))-\frac{\delta \widehat{{\mathcal F}}}{\delta m}(\hat m(t),x)\right)(m(t,x)-\hat m(t,x))dxdt\\
\ds \qquad\qquad  -\int_{\T^d} \left(G(x,m(T))-\frac{\delta \widehat{{\mathcal G}}}{\delta m}(\hat m(T),x)\right)(m(T,x)-\hat m(T,x))dx,
\end{array}
\ee
where $C$ depends on a lower bound of $D^2_{pp}H_0$ in \eqref{cond:Hcoercive}. 
By \eqref{calcul:DF} and using the fact that $m(t)$ and $\hat m(t)$ are probability measures, we have 
$$
\begin{array}{l}
\ds C^{-1}\int_{t_0}^T \int_{\T^d} (m(t,x)+\hat m(t,x))|Du-D\hat u|^2dxdt \\
\ds  \leq - \int_{t_0}^T\int_{\T^d} \left(\frac{\delta \widehat{{\mathcal F}}}{\delta m}(m(t),x) -\int_{\T^d} \frac{\delta F}{\delta m}(y, m(t),x) m(t,y)dy - \frac{\delta \widehat{{\mathcal F}}}{\delta m}(\hat m(t),x) \right)(m(t,x)-\hat m(t,x))dxdt\\
\ds  \qquad  - \int_{\T^d} \left(\frac{\delta \widehat{{\mathcal G}}}{\delta m}(m(T),x) -\int_{\T^d} \frac{\delta G}{\delta m}(y, m(T),x) m(T,y)dy - \frac{\delta \widehat{{\mathcal F}}}{\delta m}(\hat m(T),x) \right)(m(T,x)-\hat m(T,x))dx.
\end{array}
$$
As $\hat {\mathcal F}$ and $\hat {\mathcal G}$ are convex, and thus $\frac{\delta \widehat{{\mathcal F}}}{\delta m}(\hat m,x)$ and $\frac{\delta \widehat{{\mathcal G}}}{\delta m}(\hat m,x)$ are monotone, we obtain: 
$$
\begin{array}{l}
\ds C^{-1}\int_{t_0}^T \int_{\T^d} (m(t,x)+\hat m(t,x))|Du(t,x)-D\hat u(t,x)|^2dxdt \\
\ds \qquad \leq \int_{t_0}^T\int_{\T^d\times \T^d} \frac{\delta F}{\delta m}(y, m(t),x) m(t,y)(m(t,x)-\hat m(t,x))dydxdt \\
\ds \qquad\qquad +\int_{\T^d\times \T^d} \frac{\delta G}{\delta m}(y, m(T),x) m(T,y)(m(T,x)-\hat m(T,x))dydx \\
\ds \qquad\leq C  \kappa \sup_{t\in [0,T]} \|\hat m(t)-m(t)\|_{L^2(\T^d)}
\end{array}
$$
where 
$$
\kappa:=  \left(\inte \int_{t_0}^T    \left[\inte \frac{\delta F}{\delta m}(y, m(t),x) m(t,y)dy\right]^2 dt+ 
\left[\inte \frac{\delta G}{\delta m}(y, m(T),x) m(T,y)dy \right]^2\ dx \right)^{1/2}.
$$
The map $\mu(t,x):= \hat m(t,x)-m(t,x)$ solves 
$$
\partial_t \mu -\Delta \mu -\dive ( \mu D_pH_0(x, Du))= \dive( \hat m (D_pH_0(x,D\hat u)-D_pH_0(x,Du)))
$$
with initial condition $\mu(0,\cdot)=0$. So, following \cite[Lemma 7.6]{cardaliaguet2013long}, we have, for any $t\in [0,T]$,  
\begin{align*}
\| \mu(t,\cdot)\|_{L^2(\T^d)}  & \leq C \left(\int_{t_0}^T\inte \hat m^2(s,x) |D_pH_0(x,D\hat u(s,x))-D_pH_0(x,Du(s,x))|^2dxds\right)^{1/2} \\
& \leq C \left(\int_{t_0}^T\inte (\hat m(s,x)+m(s,x)) |D(\hat u-u)(s,x)|^2dxds\right)^{1/2},
\end{align*}
because $\hat m$ is bounded in $L^\infty$ by a constant which depends on the regularity of the data and of $m_0$. 
Combining the last set of inequalities, we find 
\begin{align*}
& \|\hat m-m\|_{L^2([0,T]\times \T^d)}+ \|D(\hat u-u)\|_{L^2([0,T], L^2_{m(t)+\hat m(t)}(\T^d))} \leq C\kappa.
\end{align*}
We are now in position to compare  ${\mathcal C}(u,m)$ and $ {\mathcal C}^*$:
\begin{align*}
& {\mathcal C}(u,m)  ={\mathcal C}^*+ \int_{t_0}^T\int_{\T^d} \left\{ L(x,-D_pH_0(x,D \hat u))+ F(x, \hat m(t))\right\}(m-\hat m)dxdt \\
&\;   +   \int_{t_0}^T\int_{\T^d} \left\{ L(x,-D_pH_0(x,D u))-L(x,-D_pH_0(x,D \hat u))+ F(x, m(t))-F(x, \hat m(t))\right\} m dxdt \\
&\; \leq C ( \| m-\hat m\|_{L^2([0,T]\times \T^d)} +  \|D(\hat u-u)\|_{L^2([0,T], L^2_{m(t)}(\T^d))}) \leq C\kappa.
\end{align*}
This proves the result.
\end{proof}

\section{Examples}\label{sec.examples}

Throughout this part,  we assume to fix the ideas that $H=H_0(x,p)-F(x,m)$ is separated (i.e., satisfies \eqref{eq.separate}). To simplify the expressions, we also suppose that $t_0=0$ and $G=0$. Let us recall that condition \eqref{eq:condefficient} characterizes the fact that the MFG system is globally efficient: for any initial distribution $m_0$, there exists an efficient MFG equilibrium, i.e., a solution $(u,m)$ to \eqref{MFGt} such that ${\mathcal C}(u,m)={\mathcal C}^*$.  Our first example shows that there are MFG systems which are globally efficient. However, the other examples show that this is seldom case for many standard classes of coupling functions. 

\begin{Example}\label{ex0} {\rm Let us recall that given a $C^2$ map ${\mathcal F}:\Pk\to \R$ and defining the coupling function $F$  by
$$
F(x,m):= {\mathcal F}(m)+ \frac{\delta {\mathcal F}}{\delta m} (m,x)\qquad \forall (x,m)\in \T^d\times \Pk, 
$$
the MFG system \eqref{MFGtSep} is globally efficient (Theorem \ref{thm.eff}). We now prove that,  if ${\mathcal F}$  is not affine in $m$, then $F$  genuinely depends on $m$: indeed, if $F$ does not depend on $m$, there exists a  map $f:\T^d\to \R$ with
$$
 {\mathcal F}(m)+ \frac{\delta {\mathcal F}}{\delta m} (m,x)=f(x). 
 $$
Integrating against $m$ and using Convention \eqref{eq:convention}, this implies that 
 $$
 {\mathcal F}(m)=\inte f(x)m(dx), 
 $$
which is affine in $m$. 

For instance, let $\phi:\T^d\times \T^d\to \R$ be a non vanishing map and set ${\mathcal F}$ as
$$
{\mathcal F}(m)=\int_{\T^d\times \T^d} \phi(x,y)m(dx)m (dy). 
$$
Then the coupling
\begin{align*}
F(x,m) & =  {\mathcal F}(m)+ \frac{\delta  {\mathcal F}}{\delta m} (m,x)\\
& = \inte \phi(x,y)m(dy)+\inte \phi(z,x)m(dz)-\int_{\T^d\times \T^d} \phi(z,y)m(dz)m (dy)
\end{align*}
satisfies \eqref{eq:condefficient} and depends on $m$ as soon as $\phi=\phi(x,y)$ genuinely depends on $x$ and $y$.
}\end{Example}

\begin{Example}\label{ex1}{\rm Let us now suppose that $F=F(m)$ does not depend on $x$. Then
\be\label{ilnzrsdxc}
\inte  \frac{\delta F}{\delta m}(m,y) m(dx)= \frac{\delta F}{\delta m}(m,y).
\ee
Hence the associated MFG system if globally efficient (i.e., \eqref{eq:condefficient} holds) if  only if $\frac{\delta F}{\delta m}(m,x)$ vanishes identically, in which case $F$ is constant. 

Moreover, if  $(u,m)$ is a MFG equilibrium, then we have, by \eqref{ilnzrsdxc} and the fact that $\partial_t m$ is uniformly bounded by the regularity of the data,  
\begin{align*}
\left| F(m(t_2))-F(m(t_1))\right| & =\left| \int_{t_1}^{t_2}\inte  \left(\inte  \frac{\delta F}{\delta m}(m(t),y) m(t,x)dy\right)\partial_tm(t,y)dydt\right|  \\
& \leq C(t_2-t_1)^{1/2} \left( \int_0^T\int_{\T^d} \left[\inte \frac{\delta F}{\delta m}(x, m(t),y) m(t,x)dx\right]^2dydt\right)^{1/2},
\end{align*}
so that our bound from below can be rewritten in term of the modulus of Holder continuity of $F$ along the trajectory $m$: for any $\ep\in (0,T/2)$, 
$$
{\mathcal C}( u, m)-{\mathcal C}^*\geq 
C_\ep^{-1}\left\{ \sup_{t_1\neq t_2} \frac{\left| F(m(t_2))-F(m(t_1))\right|}{(t_2-t_1)^{1/2}} \right\}^4.
$$
where the supremum is taken over $t_1, t_2\in [\ep, T-\ep]$. 
}\end{Example}

\begin{Example}\label{ex3}{\rm We now assume that $F$ derives from a potential: There exists a $C^1$ map $\Phi:\Pk\to \R$ such that $F=\delta \Phi/\delta m$. Note that, in this setting, we have 
$$
\hat {\mathcal F}(m)= \inte \frac{\delta \Phi}{\delta m}(m,x)m(dx) = 0\qquad \forall m\in \Pk.
$$
In particular, $\hat {\mathcal F}$ is convex. 
Then, by Lemma \ref{lem:calcul:DF}, 
$$
\inte \frac{\delta F}{\delta m}(x, m,y) m(dx)
= -F(y,m) \qquad \forall (y,m)\in \T^d\times \Pk.
$$
This implies that the  MFG system associated with $F$ if globally efficient  if and only if $F$ vanishes identically. 

Moreover,  if $(u,m)$ is a MFG equilibrium, 
$$
\int_0^T \inte \left[ \inte  \frac{\delta F}{\delta m}(x, m(t),y) m(t,x)dx\right]^2  dydt = 
\int_0^T \inte \left[   F(y,m(t)) \right]^2  dydt.
$$
So our estimates simply read: 
\begin{align*}
C_\ep^{-1}\left(\int_\ep^{T-\ep} \inte \left[   F(y,m(t))\right]^2  dydt\right)^2  \leq \; {\mathcal C}( u, m)-{\mathcal C}^*
\leq C\left(\int_0^T \inte \left[   F(y,m(t)) \right]^2  dydt\right)^{1/2}.
\end{align*}
}\end{Example}

\begin{Example}\label{ex2}{\rm Finally we suppose that $F$ is of the form $\ds \ F(x,m)=\int_{\T^d} \phi(x,y)m(dy) \ $
for some smooth map $\phi:\T^d\times \T^d\to \R$. Then $\ \ds \frac{\delta F}{\delta m}(x,m,y)= \phi(x,y)-\inte \phi(x,z)m(dz), \ $ so that 
$$
\inte \frac{\delta F}{\delta m}(x, m,y) m(dx)= \inte \left(\phi(x,y)- \inte \phi(x,z)m(dz)\right) m(dx). 
$$
Hence the  MFG system associated with $F$ is globally efficient if and only if 
$$
 \inte\phi(x,y)m(dx) = \int_{\T^d\times \T^d} \phi(x,z)m(dz) m(dx)\qquad \forall (y, m)\in \T^d\times \Pk,
 $$
 which implies (by choosing $m$ to be a Dirac mass), that $\phi$ does not depend on $y$. In other words, $F$ does not depend on $m$. 
 }\end{Example}


\end{document}